\documentclass[11pt]{article}
\usepackage{amssymb,latexsym,amsmath, amsthm}
\usepackage{amsfonts, graphics}
\newtheorem{theorem}{Theorem}[section]
\newtheorem{definition}[theorem]{Definition}

\newtheorem{lemma}[theorem]{Lemma}
\theoremstyle{remark}
\newtheorem*{remark}{Remark}

\begin{document}

\title{Strong Lagrangian solutions of the (relativistic)
  Vlasov-Poisson system for non-smooth, spherically symmetric data}

\author{Jacob K\"orner\\
  Institute of Mathematics\\
  Julius-Maximilians-Universit\"at W\"urzburg, Germany\\
  email: jacob.koerner@mathematik.uni-wuerzburg.de\\
  \ \\
  Gerhard Rein\\
  Department of Mathematics\\
  University of Bayreuth, Germany\\
  email: gerhard.rein@uni-bayreuth.de}

\maketitle

\begin{abstract}
  We prove a local existence and uniqueness result for the non-relativistic and
  relativistic Vlasov-Poisson system for data which need not even be continuous.
  The corresponding solutions preserve all the standard conserved quantities and
  are constant along their pointwise defined characteristic flow so that these
  solutions are suitable for the stability analysis of not necessarily smooth
  steady states. They satisfy the
  well-known continuation criterion and are global in the non-relativistic case.
  The only unwanted requirement on the data is that they be spherically
  symmetric.  
\end{abstract}

\maketitle


\section{Introduction}
\setcounter{equation}{0}
The Vlasov-Poisson system
\begin{align} \label{vlasov}
  \partial_t f+ v \cdot \partial_x f - \partial_x U \cdot \partial_v  f = 0,
\end{align}
\begin{align} \label{poisson}
  \Delta U=4\pi \rho, \quad \lim_{|x| \to \infty} U(t,x)=0,
\end{align}
\begin{align} \label{rhodef}
  \rho(t,x)= \int_{{\mathbb R}^3} f(t,x,v)\, dv      
\end{align}
describes a large ensemble of particles which interact only by the
gravitational field which they create collectively. Here $f=f(t,x,v)\geq 0$
denotes the particle density on phase space,
$t\in {\mathbb R}$, $x\in {\mathbb R}^3$, and $v\in {\mathbb R}^3$ denote
time, position, and velocity, $\rho$ is the spatial mass density induced by $f$,
and $U$ is the gravitational potential induced by $\rho$.
This system is used in astrophysics for modeling galaxies or globular
clusters, cf.\ \cite{BT}.
If the Vlasov equation \eqref{vlasov} is replaced by
\begin{align}  \label{rvlasov}
  \partial_t f+ \frac{v}{\sqrt{1+|v|^2}} \cdot \partial_x f -
  \partial_x U \cdot \partial_v  f = 0,
  \end{align}
the relativistic Vlasov-Poisson system is obtained.
Here $v$ should be thought of
as momentum; $v/\sqrt{1+|v|^2}$ is then the induced velocity of a particle
of unit mass.

An important feature of these systems is that they posses a plethora of steady
states. One way to obtain steady states is to make an ansatz
\begin{align} \label{ansatz}
f(x,v) = \phi(E(x,v)),\ 
E(x,v):= U(x) + \left\{
\begin{array}{cl}
  \frac{1}{2} |v|^2 &\ \mbox{non-relativistic case},\\
  \sqrt{1+|v|^2} &\ \mbox{relativistic case},
\end{array}
\right.
\end{align}
with some ansatz function $\phi$; $E$ is the local or particle energy in
a stationary potential $U=U(x)$. This ansatz reduces the (relativistic)
Vlasov-Poisson system to a semilinear Poisson equation for $U$, namely
\eqref{poisson} where the right hand side depends on $U$ through the ansatz
\eqref{ansatz}. We refer to \cite{RR} and the references there for sufficient
conditions on $\phi$ such that this leads to physically viable steady states
with finite mass and compact support. The classical example are the polytropic
models where
\begin{align} \label{polytr}
\phi(E) = (E_0 -E)_+^k;
\end{align}
the subscript $+$ denotes the positive part.
Here $-1< k < 7/2$ and $E_0<0$ is a cut-off energy.
One can also take sums of such ansatz functions with different cut-off energies
and/or different exponents, and if one requires $U$ to be spherically symmetric,
the ansatz may also depend on the particle angular momentum $L:=|x\times v|$.
The important point for the present paper is that for these steady states $f$
need not be smooth and not even continuous.
If one wants to investigate the stability of such a steady state
a natural class of perturbations are the dynamically accessible ones, where the
stationary particle distribution is rearranged via a measure-preserving
homeomorphism of phase space caused for example by the action of some exterior,
perturbing force. The resulting, dynamically accessible data are then in general
as regular or irregular as the original steady state. For example, if we pick
$k=0$ in \eqref{polytr}, then the distribution function of both the steady state
and its perturbation attains only the values $0$ and $1$ and is discontinuous.

It is desirable to have an existence and uniqueness result for the
time dependent problem for such data, where the resulting solutions should
preserve all the conserved quantities like the total energy and the so-called
Casimir functionals, since these are used in the stability analysis.
In addition, the characteristic flow corresponding
to the Vlasov equation should exist and $f$ should be constant along this flow;
this property is more important in a stability analysis (and elsewhere)
than the Vlasov equation itself. We refer to
\cite{GuRe2001,GuRe2007,GuoLin,HR,LeMeRa2,LeMeRa3,Rein07}
and the references there for
stability results for the (relativistic) Vlasov-Poisson system.
In the present paper we provide a local existence and uniqueness result as
specified above. The characteristic flow of the corresponding Vlasov equation
will be defined pointwise on phase space, and $f$ will be constant along it;
we call such solutions
{\em strong Lagrangian solutions}. For the data we require that
$\mathring{f} = f_{|t=0}$
is a non-negative,
bounded, and measurable function with compact support, and in addition,
that $\mathring{f}$ is spherically symmetric; a function
$g\colon {\mathbb R}^3\times {\mathbb R}^3 \to {\mathbb R}$ is
{\em spherically symmetric}, if $g(Ax,Av)=g(x,v)$ for all
$x,v \in {\mathbb R}^3$ and
$A\in \mathrm{SO}(3)$. In passing we note that our result answers a
question which was left open in the stability analysis \cite{HR}.
The symmetry assumption is of course undesirable, and
it is an open problem, how far one can relax this assumption without loosing
any of the properties of the solution. It is also an open problem whether one
can preserve these properties for not necessarily bounded data, such as would
arise by perturbing polytropic steady states with $-1<k<0$.

There exists an extensive literature concerning the initial value problem
for the Vlasov-Poisson system, and a bit less for its relativistic version,
and to put the present paper into context we recall some of it.
For the non-relativistic version,
Batt \cite{B} proved local existence and uniqueness of smooth solutions
together with a continuation criterion, and he used the latter to obtain
global existence for smooth, spherically symmetric data. The latter result is
known to be false for the relativistic version, cf.~\cite{GS}.
For the non-relativistic version smooth solutions exist globally also for
non-symmetric data, as was shown by Pfaffelmoser \cite{Pf} and simultaneously
but independently by Lions and Perthame \cite{LP}, cf.\ also \cite{Rein07}.
Global weak solutions for the non-relativistic system, which are neither known
to be unique nor to preserve the usual conserved quantities, were obtained for
example in \cite{Ars, HoHu}. More recently, Lagrangian flows for non-smooth
vector fields have been investigated and used to construct Lagrangian
solutions of the Vlasov-Poisson system for $L^1$ data, cf.\ \cite{BBC}
and the references there. However, \cite{BBC} considers the non-relativistic,
repulsive case of the Vlasov-Poisson system where the sign of the right hand
side in the Poisson equation \eqref{poisson} is reversed. We do not know
whether these results can be extended to the attractive case stated above or
to the relativistic one. The relation
of the flow of ordinary differential equations with coefficients in Sobolev
spaces to linear transport equations like the Vlasov equation was studied in
the seminal paper \cite{DiPL}.
It should be emphasized that the results and techniques in \cite{BBC,DiPL}
are much more far reaching and sophisticated than the present investigation
and in particular do not rely on any symmetry assumption. Indeed, the main point
of the present investigation is to show that for the price of assuming
spherical symmetry, Lagrangian solutions with all the desired properties,
in particular, with a pointwise defined characteristic flow,
can be obtained by quite elementary methods for both the non-relativistic
and the relativistic Vlasov-Poisson system; in passing we note that all our
results hold equally well for the repulsive case mentioned above.

In the next section we state our results, and the proofs are given in
Section~\ref{proofs}.
The present paper is based on the first authors master thesis \cite{K}.
\section{Main results} \label{results}
\setcounter{equation}{0}
We start by making precise our solution concept; throughout
the paper integrals without an explicitly denoted domain of integration
extend over ${\mathbb R}^3$ or ${\mathbb R}^6$.
\begin{definition}\ \label{def_strongL}
  A measurable function
  $f\colon [0,T[\times {\mathbb R}^6 \to {\mathbb R}$ with $T>0$ is a
  {\em strong Lagrangian solution} of the non-relativistic or relativistic
  Vlasov-Poisson system iff:
  \begin{itemize}
  \item[(i)]
    The induced mass density
    \[
    \rho_f(t,x) =\rho(t,x):= \int f(t,x,v)\, dv
    \]
     and the induced gravitational field
     \[
     F_f(t,x) =F(t,x) :=\int \frac{x-y}{|x-y|^3} \rho (t,y) \, dy
     \]
     exist for all $(t,x) \in [0,T[ \times {\mathbb R}^3$, and $F$ is continuous
     and Lipschitz continuous in $x$, locally uniformly in $t$, i.e.,
     for every $0<T'<T$ there exists $L>0$ such that for all
     $t \in [0,T']$ and  $x, x'\in {\mathbb R}^3$,
     \[
     |F(t,x) - F(t,x')| \leq L\, |x-x'|.
     \]
   \item[(ii)]
     $f$ is constant
     along its characteristics, i.e., for all
     $(t,z)\in [0,T[ \times {\mathbb R}^6$,
     the mapping $s \mapsto f(s,Z(s,t,z))$ is constant, where
     $s\mapsto Z(s,t,z)=(X,V)(s,t,x,v)$ is the solution of the
     characteristic system
     \begin{align} \label{chars}
       \dot x = v \ \mbox{or}\ \dot x = \frac{v}{\sqrt{1+|v|^2}},\ \qquad
       \dot v = - F(s,x)
     \end{align}
     with $Z(t,t,z) = z = (x,v)$.
\end{itemize}  
\end{definition}
The gravitational field $F$ defined in part~(i) is the gradient
of the potential determined by \eqref{poisson}, and the conditions on $F$
guarantee the existence of the characteristic flow used in part~(ii) of the
definition, see also Lemma~\ref{lemma_char} below.
Formally, the definition can be
relaxed by replacing the assumptions on the field $F$ by the properties of the
induced flow, obtained in  Lemma~\ref{lemma_char}.
We also note that no symmetry
assumption enters in this definition.

For a measurable, bounded, and compactly supported state
$g \colon {\mathbb R}^6 \to [0,\infty[$ we define its kinetic
and potential energies as
\[
E_{\mathrm{kin}}(g):=\frac{1}{2} \iint |v|^2 g(x,v)\, dv\, dx
\]
or
\[
E_{\mathrm{kin}}(g):=\iint \sqrt{1+|v|^2} g(x,v)\, dv\, dx,
\]
\[
E_{\mathrm{pot}}(g):=-
\frac{1}{2} \iiiint \frac{g(x,v)\, g(y,w)}{|x-y|} dv\, dw\, dx\, dy ,
\]
and a Casimir functional is defined as
\[
\mathcal{C}(g) := \iint \Phi\big( g(x,v) \big)\, dv \, dx,
\]
where $\Phi\colon {\mathbb R} \to {\mathbb R}$ is continuous with $\Phi(0)=0$ 
\begin{theorem} \label{main}
  Let $\mathring{f} \colon {\mathbb R}^6 \to [0,\infty[$ be measurable, bounded,
  compactly supported and spherically symmetric. Then there exists
  a unique, spherically symmetric, strong Lagrangian solution
  $f \colon [0,T[\times {\mathbb R}^6 \to [0,\infty[$
  of the (relativistic) Vlasov-Poisson
  system with $f(0)=\mathring{f}$. If $T>0$ is chosen maximal, then in the
  non-relativistic case, $T=\infty$.
  In the relativistic case, $T=\infty$ if
  \[
  \sup \big\{|v| \mid (x,v) \in  \mathrm{supp}\, f(t),\
  0 \leq t <T \big\} < \infty.
  \]
  The energy and all Casimir functionals are conserved, i.e.,
  for all $t \in [0,T[$,
  \[
  E_{\mathrm{kin}}\big(f(t)\big) + E_{\mathrm{pot}}\big(f(t)\big) =
  E_{\mathrm{kin}}\big(\mathring{f}\,\big) +
  E_{\mathrm{pot}}\big(\mathring{f}\,\big), \quad
  \mathcal{C}\big(f(t)\big) = \mathcal{C}\big(\mathring{f}\,\big).
  \]
\end{theorem}
Some additional properties of the solution $f$ which come out of the proof
will be listed below.
\section{Proofs} \label{proofs}
\setcounter{equation}{0}
\subsection{The characteristic flow and the Vlasov equation}
We recall the relevant properties of the flow induced by \eqref{chars}.
\begin{lemma} \label{lemma_char}
  Let $F \colon [0,T[ \times \mathbb{R}^3 \to \mathbb{R}^3$
  be continuous and Lipschitz continuous
  with respect to $x$, locally uniformly in $t$. Then the following holds:
  \begin{itemize}
  \item[(a)]
    For every $t \in [0,T[$ and
    $z = (x,v) \in {\mathbb R}^3 \times {\mathbb R}^3$
    there exists a unique solution $[0,T[ \ni s \mapsto Z(s,t,z)$ of
    \eqref{chars} with $Z(t,t,z)=z$.
    The flow $Z$ is continuous on $[0,T[ \times [0,T[ \times {\mathbb R}^6$
    and Lipschitz continuous with respect to $z$, locally uniformly
    in $s$ and $t$.        
  \item[(b)] For every $s,t \in [0,T[$, the mapping
    $Z(s,t,\cdot): {\mathbb R}^6 \to {\mathbb R}^6$ is measure preserving, i.e.,
    \[
    |\det \partial_z Z(s,t,z)| = 1 \
    \mbox{for almost every}\ z\in {\mathbb R}^6,
    \]
    and one-to-one and onto
    with inverse $Z^{-1} (s,t,\cdot) = Z(t,s,\cdot)$.
  \item[(c)] For any measurable function
    $\Phi \colon {\mathbb R}^6 \to {\mathbb R}$,
    any measurable set $D \subset {\mathbb R}^6$, and any $s,t\in [0,T[$
    the change-of-variables formula holds:
    \[
    \int_{Z(s,t,D)} \Phi(z)\, dz = \int_D \Phi(Z(s,t,z))\, dz .
    \]
  \item[(d)]
    If $F$ in addition is spherically symmetric, i.e., $F(t,A x)=A\,F(t,x)$
    for all $t\in [0,T[$, $x\in {\mathbb R}^3$,
    and $A\in \mathrm{SO}(3)$, then so is
    $Z=(X,V)$,
    i.e., $(X,V)(s,t,A x,A v)=(A\,X,A\,V)(s,t,x,v)$ for all
    $s, t\in [0,T[$, $x, v\in {\mathbb R}^3$, and $A\in \mathrm{SO}(3)$.
  \end{itemize}
\end{lemma}
\begin{proof}
  Most of parts (a) and (b) is standard ODE theory. The fact that for
  $s$ and $t$ fixed, $Z(s,t,\cdot)$ is Lipschitz implies that the
  derivative $\partial_z Z(s,t,z)$ exists for almost all $z$;
  the exceptional set of measure zero may depend on $s$ and $t$, but
  this causes no problems.
  To prove the assertion on the functional determinant,
  let $J \in C_c^\infty ({\mathbb R}^3)$ be a smooth, compactly
  supported function with
  $\int J= 1$, i.e., a Friedrichs mollifier.
  For $\epsilon> 0,$ we define 
  $J_\epsilon := \epsilon^{-3} J(\cdot/\epsilon)$ and the smoothed field
  $F_\epsilon(t) := J_\epsilon * F(t)$ where $F(t)=F(t,\cdot)$ for $t \in [0,T[$.
  The corresponding flow $Z_\epsilon$ is differentiable with respect to $z$ with
  \[
  \det \partial_z Z_\epsilon(s,t,z)=1,\ s, t \in [0,T[,\ z\in {\mathbb R}^6,
  \]
  since the vector field generating this flow is divergence free on
  ${\mathbb R}^3\times {\mathbb R}^3$,
  cf.\ for example \cite[Lemma~1.2]{Rein07}. Moreover,    
  $Z_\epsilon(s,t,z) \to Z(s,t,z)$ for $\epsilon \to 0$,
  uniformly in $z$ and locally
  uniformly in $s$ and $t$. Now let
  $\phi \in C_c^\infty ({\mathbb R}^6)$ denote any test
  function. Then the change-of-variables formula for Lipschitz continuous
  transformations---cf.\ \cite[*263F Corollary]{F}---and the above convergence
  imply that
  \begin{align*}
    \int \phi (z)  |\det\partial_z Z(s,t,z)|\, dz 
    &= \int \phi (Z(t,s,z))\, dz
    = \lim_{\epsilon \to 0 } \int \phi (Z_\epsilon(t,s,z))\,dz \\ 
    &=\lim_{\epsilon \to 0 } \int \phi (z) |\det\partial_z Z_\epsilon (s,t,z)|\, dz 
    = \int \phi (z)\, dz.
  \end{align*}
  Hence, $|\det \partial_z Z(s,t,z)| =1$ for almost every $z$. Combining this
  again with \cite[*263F Corollary]{F} yields part (c). Part (d) follows by
  uniqueness.
\end{proof}

Given a field $F$ as specified in the previous lemma and initial data we can
solve the corresponding Vlasov equation.
\begin{lemma} \label{lemma_vlasov} 
  Let $F$ be as in Lemma~\ref{lemma_char}, $Z$ the flow obtained there, and
  let $\mathring{f} \colon {\mathbb R}^6\to {\mathbb R}$ be measurable, bounded,
  and compactly supported,
  and define $f(t,z) :=  \mathring{f} (Z(0,t,z))$ for all $t \in [0,T[$
  and $z\in {\mathbb R}^6$.
  Then the following holds: 
  \begin{itemize}
  \item[(a)]
    $f$  is constant along  solutions of \eqref{chars}, and
    $f(0)=  \mathring{f} $. 
  \item[(b)]
    For every $t \in [0,T[$ and $p \in [1,\infty]$,
    $\mathrm{supp}\, f(t) =Z(t,0,\mathrm{supp}\, \mathring{f})$, and
    $\Vert f(t)\Vert_p = \Vert\mathring{f}\Vert_p$; here $\Vert\cdot\Vert_p$
    is the $L^p$ norm on ${\mathbb R}^6$, and
    $f(t)=f(t,\cdot)$.
  \item[(c)]
    $f \in C\big([0,T[; L^1({\mathbb R}^6)\big)$.
  \item[(d)]
    If $F$ and $\mathring{f}$ are spherically symmetric,
    then so is $f(t)$ for every $t \in [0,T[$.
\end{itemize}
\end{lemma}
\begin{proof}
  With the possible exception of part (c) all of this is quite obvious by
  the definition of $f$ and Lemma~\ref{lemma_char}. As to part (c),
  we first notice that $f(t)$ is measurable, since $\mathring{f}$
  is measurable and
  $Z(t,0,\cdot)$ is one-to-one and Lipschitz. Now let
  $\epsilon >0$ be arbitrary and choose
  $g \in C_c^\infty ({\mathbb R}^6)$ such that
  $\Vert\mathring{f}  - g\Vert_1 <\epsilon$.
  Then for any $t,t' \in [0,T[$,
  \begin{align*}
    \Vert f(t) - f(t') \Vert_1
    &\leq \int \vert  \mathring{f}  (Z(0,t,z)) - g(Z(0,t,z)) \vert\, dz\\
    & \quad {}+ \int \vert  \mathring{f}  (Z(0,t',z)) - g(Z(0,t',z)) \vert\, dz
    \\
    & \quad {}+ \int \vert  g (Z(0,t,z)) - g(Z(0,t',z)) \vert \, dz
    \\
    & \leq 2 \epsilon + \int_{B_R} \vert g (Z(0,t,z)) - g(Z(0,t',z)) \vert\, dz,
  \end{align*}
  where $B_R\subset {\mathbb R}^6$
  is a sufficiently large ball about the origin,
  and the assertion follows by continuity of $g$ and $Z$.
\end{proof}
\subsection{Local existence}
In this section we prove the local existence part of Theorem~\ref{main}.
To this end we consider the following iteration scheme which is essentially
the same as in \cite[Thm.~1.1]{Rein07}.

We define the $0$th iterate of the field as $F_0(t,x)=0$ for all
$t\in [0,\infty[$ and $x\in {\mathbb R}^3$.
Assume that for some $n\in {\mathbb N}_0$
a field $F_n\colon [0,\infty[ \times {\mathbb R}^3 \to {\mathbb R}^3$
is already defined
which has the following properties.
    
\noindent
{\em Field properties:} $F_n$ is continuous in $t$ and $x$,
Lipschitz continuous in $x$, locally uniformly in $t$,
bounded on $[0,T']\times {\mathbb R}^3$ for any $T'>0$,    
and spherically symmetric.

The field $F_0$ clearly has these properties. Lemma~\ref{lemma_char} yields
a corresponding flow $Z_n$, and Lemma~\ref{lemma_vlasov} yields the $n$-th
iterate $f_n$. We complete one iteration step by defining $\rho_n:= \rho_{f_n}$
and $F_{n+1}:=F_{f_n}$, cf.\ Definition~\ref{def_strongL}~(i).
Local existence now follows in three steps.

\noindent
{\em Step 1.}
In this step we prove that the iteration is well defined. Let
\[
P_n (t) := \sup \bigl\{|V_{n}(s,0,z)| \mid z \in \mathrm{supp}\,\mathring{f},\
0 \leq s \leq t  \bigr\},
\]
and pick $\mathring R, \mathring P >0$ such that
$\mathrm{supp}\,  \mathring{f}  \subset B_{\mathring R} \times B_{\mathring P}$;
the latter balls
are now in  ${\mathbb R}^3$. Then
\[
f_n(t,x,v)=0 \ \mbox{for}\ |v| \geq P_n(t)\ \mbox{or}\ |x| \geq
\mathring R + \int_{0}^t P_n(s)\, ds,
\]
\[
\rho_n(t,x)=0 \ \mbox{for}\ |x| \geq
\mathring R + \int_{0}^t P_n(s)\, ds,
\]
and
\[
\Vert \rho_n (t) \Vert_\infty
\leq \frac{4 \pi}{3}\Vert  \mathring{f}  \Vert_\infty
P_n(t) ^3.
\]
Now we recall that for any $\rho\in L^1\cap L^\infty({\mathbb R}^3)$
the field generated by $\rho$ satisfies the estimate
\begin{align}\label{Fest_gen}
\Vert F_\rho\Vert_\infty
\leq 3 (2\pi)^{2/3} \Vert\rho\Vert_1^{1/3}\Vert\rho\Vert_\infty^{2/3},
\end{align}
cf.\ for example \cite[Lemma~P1]{Rein07}.
Since $f_n\in C([0,\infty[;L^1({\mathbb R}^6))$
and hence $\rho_n\in C([0,\infty[;L^1({\mathbb R}^3))$,
\eqref{Fest_gen} implies that
$F_{n+1}$ is continuous in $t$. Moreover, for all $t\geq 0$,        
\begin{align} \label{Fn_bound}
  \Vert F_{n+1}(t) \Vert_\infty \leq 3 (2 \pi)^{2/3} \Vert\rho_n(t)\Vert_1^{1/3}
  \Vert\rho_n (t)\Vert_\infty ^{2/3}
\leq C_{\mathring{f}} P_n(t)^2, 
\end{align}
where
\begin{align} \label{cfdef}
  C_{\mathring{f}} :=
  4 \cdot 3^{1/3} \pi^{4/3} \Vert\mathring{f} \Vert_1^{1/3}
  \Vert\mathring{f} \Vert_\infty^{2/3},
\end{align}
in particular, the field $F_{n+1}$ is bounded, locally uniformly in $t$.
The spherical symmetry is inherited by $f_n$,
and hence by $\rho_n$ and $F_{n+1}$,
and to see that $F_{n+1}$ has the field properties
formulated above, it remains to show its Lipschitz property;
this is the first instance where we need to exploit the symmetry assumption.
Because of the latter,
\begin{align} \label{symmfield}
  F_{n+1}(t,x) = G_{n+1}(t,r) \frac{x}{r},
  \ \mbox{where}\
  G_{n+1}(t,r) := \frac{4\pi}{r^2} \int_0^r \rho_n (t,s)\, s^2 ds;
\end{align}
here $r=|x|$, and we identify $\rho_n(t,x)$ and $\rho_n(t,r)$.
For any $t\geq 0$ and $0<u<r$, 
\begin{align} \label{lipG}
  |G_{n+1}(t,r)- G_{n+1}(t,u)|
  \leq &\frac{4\pi}{r^2} \int_u^r \rho_n (t,s)\,s^2 ds\nonumber\\
  &{}+  4\pi \Big\vert\frac{1}{r^2} - \frac{1}{u^2}\Big\vert
  \int_0^u \rho_n (t,s)\,s^2 ds \nonumber \\
  \leq &\frac{20\pi}{3} \Vert\rho_n(t)\Vert_\infty  |r-u| . 
\end{align}
The required Lipschitz property of $F_{n+1}$ follows from \eqref{lipG}.

\noindent
{\em Step 2:} We establish bounds that are uniform in $n$. The definition of
$P_n$ and \eqref{Fn_bound} imply that for $n\in {\mathbb N}_0$ and $t\geq 0$,
\[
P_{n+1}(t) \leq \mathring P + \int_0^t \Vert F_{n+1}(s)\Vert_\infty ds
\leq \mathring P + C_{\mathring{f}} \int_0^t P_n(s)^2 ds.
\]
If we drop the subscripts of $P$ and replace $\leq$ by $=$, we obtain
an integral equation the unique, maximal solution of which is
\begin{align} \label{Qdef}
  Q \colon [0, (\mathring P C_{\mathring{f}})^{-1}[ \to [0, \infty[, \
  t \mapsto \frac{\mathring P}{1- \mathring P C_{\mathring{f}} t },
\end{align}
and a straight forward induction argument shows that
$P_n(t) \leq Q(t)$  for all $n\in {\mathbb N}_0$ and
$t\in  [0, \delta_0[$, where $\delta_0:=(\mathring P C_{\mathring{f}})^{-1}$.

\noindent
{\em Step 3:} We show that on any compact subinterval
$[0,\delta]\subset [0,\delta_0[$ the iteration sequence converges
in a suitable sense, and its limit is a strong Lagrangian solution.

Using the characteristic system and observing that in view of the uniform bounds
and \eqref{lipG} the fields $F_n$ are Lipschitz in $x$ uniformly on
$[0,\delta]$ and uniformly in $n$ we find that
\begin{align*}
  |Z_{n+1}(t,0,z)- Z_{n}(t,0,z)|
  \leq & C \int_0^t |Z_{n+1}(s,0,z) - Z_{n}(s,0,z)|\, ds\\
  & {}+ \int_0^t\Vert G_{n+1} (s) -  G_n(s) \Vert_\infty ds;
\end{align*}
here and in what follows $C$ denotes a positive constant
which may only depend on
$\mathring{f}$ and $\delta_0$ and which may change its value from line to line.
For the relativistic case we note that the map
$v\mapsto v/\sqrt{1+|v|^2}$ is Lipschitz continuous.
By Gronwall,
\begin{align} \label{DZn < DGn}
  |Z_{n+1}(t,0,z)- Z_{n}(t,0,z)| \leq C
  \int_0^t\Vert G_{n+1} (s) -  G_n(s) \Vert_\infty ds. 
\end{align}
The crucial point is to estimate the latter difference,
and this is also the point
where the symmetry assumption enters most strongly, cf.~\eqref{symmfield}.
For $t\in [0,\delta]$, $r\geq 0$ and $n\in {\mathbb N}$
we first note that by the uniform estimate on $\rho_n$,
\begin{align}
DG_n(t,r) &:= \vert G_{n+1} (t,r) - G_n(t,r) \vert\nonumber \\
&= \frac{4 \pi}{r^2}
\Big\vert \int_0^r \big(\rho_{n} (t,s) -\rho_{n-1} (t,s)\big)\, s^2 ds \Big\vert
\leq C\, r. \label{Force < Cr}
\end{align}
On the other hand, denoting $z=(y,v)\in {\mathbb R}^3\times {\mathbb R}^3$,
\begin{align*}
\{ Z_{n}(0,t,z) \mid |y| \leq r \} 
&= \{ \tilde z \in {\mathbb R}^6
\mid \exists \, z \in B_r \times {\mathbb R}^3\
\mbox{such that}\ Z_{n}(t,0,\tilde z) = z \}
\\
&= \{ \tilde z \in {\mathbb R}^6 \mid |X_n (t,0,\tilde z)| \leq r \}.
\end{align*}
Hence we can rewrite the modulus of the field as
\begin{align} \label{G_rewr}
  G_{n+1}(t,r) 
  &= \frac{1}{r^2} \int_{\vert y \vert \leq r} \rho_{n}(t,y)\, dy 
  = \frac{1}{r^2} \int_{\{z \in {\mathbb R}^6 \mid |y| \leq r\} }
  \mathring{f}(Z_{n}(0,t,z))\, dz
  \nonumber \\
  &=\frac{1}{r^2} \int_{Z_{n}(0,t,B_r\times{\mathbb R}^3)}  \mathring{f}(z)\, dz 
  = \frac{1}{r^2} \int_{\{ z \in {\mathbb R}^6 \mid |X_{n} (t,0, z)| \leq r\} }
  \mathring{f}(z)\, dz,	
\end{align}
where we have used the change-of-variables formula from
Lemma~\ref{lemma_char}~(c).
This implies that
\begin{align}\label{DForce < vol Dn}
DG_n(t,r) 
&\leq \frac{1}{r^2} \Big \vert \int_{\{ z \in {\mathbb R}^6 \mid |X_n (t,0,z)| \leq r\} }
\mathring{f} (z)\, dz
- \int_{\{ z \in {\mathbb R}^6 \mid |X_{n-1} (t,0, z)| \leq r\} }
\mathring{f} (z)\, dz \Big \vert
\nonumber \\
& \leq \frac{1}{r^2} \Vert \mathring{f} \Vert_\infty
\lambda (D_n), 
\end{align}
where $\lambda$ denotes the Lebesgue measure, and
\begin{align*}
  D_n :=  \Bigl\{ z \in \mathrm{supp}\, \mathring{f} \mid &\;
  |X_n(t,0,z)| \leq r < |X_{n-1} (t,0,z)| \\
& \; \lor \; |X_{n-1}(t,0,z)| \leq r < |X_{n} (t,0,z)|\Bigr\}.
\end{align*}
Defining
\begin{align*}
d_n := \sup_{z \,\in\, \mathrm{supp}\,   \mathring{f} } |X_n(t,0,z) - X_{n-1} (t,0,z)|,
\end{align*}
we observe that
\begin{align*}
\lambda (D_n)
&\leq
\lambda\Bigl(\Bigl\{z\in\mathrm{supp}\, \mathring{f} \mid
|X_n(t,0,z)| \leq r < d_n + |X_n(t,0,z)| \\
& \qquad \qquad \qquad \qquad
\lor |X_{n-1}(t,0,z)| \leq r < d_n + |X_{n-1}(t,0,z)| \Bigr\}\Bigr) \\
&\leq \,
\lambda\Bigl(\Bigl\{z\in\mathrm{supp}\,\mathring{f} \mid
|X_n(t,0,z)| \leq r < d_n + |X_n(t,0,z)|\Bigr\}\Bigr) \\
& \quad {}+
\lambda\Bigl(\Bigl\{z\in\mathrm{supp}\,\mathring{f} \mid
|X_{n-1}(t,0,z)| \leq r < d_n + |X_{n-1}(t,0,z)|\Bigr\}\Bigr) \\
& =: \lambda (D^1_n) +  \lambda (D^2_n).
\end{align*}
We use the fact that the characteristic flow is measure preserving
to eliminate the $X_n$-terms:
\begin{align*}
  \lambda (D^1_n)
  &= \lambda (Z_n(t,0,D_n^1)) \\
  & = \lambda\Bigl( \Bigl\{Z_n(t,0,z) \mid\!
  z \in \mathrm{supp}\,\!\mathring{f}  \land
  |X_n(t,0,z)| \leq r < d_n + |X_n(t,0,z)|  \Bigr\}\Bigr) \\
  & = \lambda\Bigl( \Bigl\{ (y,v) \in  Z_n(t,0,\mathrm{supp}\,\mathring{f})
  \mid
  |y| \leq r < d_n + |y| \Bigr\}\Bigr) \\
  & \leq \lambda \Big(  \{ y \in B_R \mid r - d_n < |y|\leq r  \}
  \times B_R \Big) \\
  & \leq C \, \big( r^3 -(r-d_n)_+^3 \big),
\end{align*}
where we recall that by our uniform estimates,
$Z_n(t,0,\mathrm{supp}\,\mathring{f})\subset B_R \times B_R$
with some radius $R>0$
which is uniform in $n\in {\mathbb N}_0$ and $t\in [0,\delta]$.
The same result holds for $D_n^2$;
we just have to replace $Z_{n}$ by $Z_{n-1}$. Hence
\begin{align}\label{vol Dn < dn...}
  \lambda (D_n) \leq  C \, \big( r^3 -(r-d_n)_+^3 \big).	
\end{align}
If $r \leq d_n$, we use \eqref{Force < Cr} to find that
\begin{align*}
  DG_n(t,r) \leq C\, r \leq C\, d_n .
\end{align*}
If $r>d_n$, \eqref{DForce < vol Dn} and \eqref{vol Dn < dn...} imply that
\begin{align*}
  DG_n(t,r) \leq \frac{C}{r^2} (d_n^3 + 3 r^2 d_n) \leq C\, d_n.
\end{align*}
Combining both results, we see that
\begin{align*}
\Vert G_{n+1} (t) - G_n (t) \Vert_\infty
&\leq C \sup_{z \, \in \, \mathrm{supp}\, \mathring{f}} |X_n(t,0,z) - X_{n-1}(t,0,z)| \\
& \leq C \sup_{z \, \in \,  \mathrm{supp}\, \mathring{f} } |Z_n(t,0,z) - Z_{n-1}(t,0,z) |,	
\end{align*}
and together with \eqref{DZn < DGn} we finally arrive at the estimate
\begin{align}\label{DFn_est}
\Vert F_{n+1} (t) -  F_n(t) \Vert_\infty
\leq C \int_0^t \Vert F_{n} (s) -  F_{n-1}(s) \Vert_\infty ds,
\end{align}
which holds for all $n\in {\mathbb N}$ and $t\in [0,\delta]$. This implies that
$(F_n)$ is a Cauchy sequence in the space
$C([0,\delta]; L^\infty({\mathbb R}^3))$.
Thus there exists a limiting field
$F\colon [0,\delta]\times {\mathbb R}^3 \to {\mathbb R}^3$ such that
$ F_n \to F $ uniformly on $[0,\delta] \times {\mathbb R}^3$.
The field is bounded
and continuous. By the previous two steps, $F_n$ is Lipschitz continuous in
$x$, uniformly in $t\in [0,\delta]$ and in $n\in {\mathbb N}$.
Hence $F$ is Lipschitz continuous in
$x$, uniformly in $t\in [0,\delta]$. Since $\delta < \delta_0$ is arbitrary,
the field $F$ exists and has the desired properties on $[0,\delta_0[$.
Lemma~\ref{lemma_char} yields the corresponding flow $Z$, and
$Z_n \to Z$ uniformly on
$[0,\delta] \times [0,\delta] \times {\mathbb R}^6$ for all
$\delta < \delta_0$. If we define $f(t,z) = \mathring{f}(Z(0,t,z))$ according to
Lemma~\ref{lemma_vlasov} it remains to show that $F$ is indeed the field
induced by $f$.

By Lebesgue's dominated convergence theorem and \eqref{G_rewr},
\begin{align*}
  G (t,r) 
  &= \lim_{n\to \infty} G_n(t,r) 
  = \frac{1}{r^2} \lim_{n\to \infty}
  \int_{\{ z \in {\mathbb R}^6 \mid |X_{n-1} (t,0, z)| \leq r\} }  \mathring{f}  (z)\, dz \\
  &=  \frac{1}{r^2} \int_{\{ z \in {\mathbb R}^6 \mid |X (t,0, z)| \leq r\} }
  \mathring{f}  (z)\, dz
  = \frac{1}{r^2}  \int_{\{z \in {\mathbb R}^6 \mid |y| \leq r\} }
  \mathring{f}  (Z(0,t,z))\, dz,
\end{align*}
which implies that
\begin{align*}
  F(t,x) =  G (t,r) \frac{x}{r}
  = \iint \frac{x-y}{|x-y|^3} f(t,y,v)\, dv \, dy.
\end{align*}
Hence $f$ has all the properties of a strong Lagrangian solution on
$[0,\delta[$, and of course $f(0)=\mathring{f}$.
\subsection{Uniqueness}
Assume that we have two spherically symmetric,
strong Lagrangian solutions to the same initial data
with fields $F$ and $\tilde F$. Then on any time interval $[0,\delta]$
where both are defined we can treat the difference $F-\tilde F$ exactly
as we treated $F_{n+1}-F_n$ above, in particular  $F-\tilde F$
must satisfy the analogue of \eqref{DFn_est} which implies that
the two fields and hence the two solutions are equal on $[0,\delta]$.
\begin{remark}
  The argument above yields uniqueness only within
  the class of spherically symmetric, strong Lagrangian solutions, which is
  sufficient for what follows below. However, uniqueness holds also within the
  class of (not necessarily symmetric)
  strong Lagrangian solutions for data $\mathring{f}$ which are measurable,
  bounded, and compactly supported. To see this, we observe that for such
  solutions the induced spatial density is again
  bounded and compactly supported with respect to $x$, locally uniformly in
  $t$. Moreover, such solutions are easily seen to be weak solutions, and the
  uniqueness results in \cite{Lo,Ro} apply, at least in the non-relativistic
  case; it seems reasonable to expect this to remain true also in the
  relativistic case.
\end{remark}
\subsection{Continuation and global existence}
We can extend the unique, strong Lagrangian solution to its maximal interval
of existence $[0,T[$. Assume that
\begin{align}\label{Past}
  P^\ast := \sup \big\{|v| \mid (x,v) \in \mathrm{supp}\, f(t),
  \ 0 \leq t <T \big\}
  < \infty .
\end{align}
For any $t_0\in [0,T[$,
\[
\Vert f(t_0) \Vert_\infty = \Vert  \mathring{f}  \Vert_\infty, \quad \Vert f(t_0)
\Vert_1 = \Vert  \mathring{f}  \Vert_1,
\] 
and hence
$C_{f(t_0)} = 4 \cdot 3^{1/3} \pi ^{4/3} \Vert f(t_0) \Vert_1^{1/3}
\Vert f(t_0) \Vert_\infty ^{2/3} = C_{\mathring{f}}$,
cf.\ \eqref{cfdef}.
We define  $\delta_0^\ast := (P^\ast C_{f(t_0)} )^{-1} = (P^\ast C_{\mathring{f}})^{-1}$.
Arguing exactly as before we obtain a strong Lagrangian solution
on the interval $[t_0,t_0+\delta_0^\ast[$ to the initial data $f(t_0)$
prescribed at $t_0$. By uniqueness, this solution must coincide with
$f$ as long as both exist. If $T$ were finite, this would extend the maximal
solution beyond $T$, provided we choose $t_0$ close enough to $T$;
note that $\delta_0^\ast$ is independent of $t_0$.

The continuation criterion which is now established applies to both the
non-relativistic and the relativistic case, and in the former we can
verify that \eqref{Past} indeed holds and hence $T=\infty$.
To this end we observe that the corresponding argument of Horst \cite{Ho}
applies to strong Lagrangian solutions, see also \cite[Thm.~1.4]{Rein07}.
\subsection{Conservation laws}
The conservation of Casimir functionals is a direct consequence of the
change-of-variables formula in Lemma~\ref{lemma_char}~(c) and the definition
of a strong Lagrangian solution.

Next we prove conservation of energy for the non-relativistic case,
the relativistic case being completely analogous.
We use the fact that the flow is measure preserving,
cf.~Lemma~\ref{lemma_char}~(b),
together with the fact that $f$ is constant along
the flow and the fundamental theorem of
calculus, and we recall the notation
$z=(x,v)\in {\mathbb R}^3\times{\mathbb R}^3$
and analogously,
$\tilde z=(\tilde x,\tilde v)\in {\mathbb R}^3\times{\mathbb R}^3$. Then 
\begin{align*}
  &
  2 E_{\mathrm{kin}}(f(t))  + 2 E_{\mathrm{pot}}(f(t))\\
  &= \int |v|^2 f(t,z)\, dz -
  \iint \frac{f(t,z)\,f(t,\tilde{z})}{|x - \tilde{x}|} dz \, d\tilde{z} \\
  &= \int |V(t,0,z)|^2 f(t, Z(t,0,z))\, dz 
  - \iint \frac{f(t, Z(t,0,z))\,f(t, Z(t,0,\tilde{z}))}
  {|X(t,0,z) -X(t,0,\tilde z)|} dz \, d\tilde{z} \\
  &= \iint_0^t \frac{d}{ds}\Big(|V(s,0,z)|^2
  f(s,Z(s,0,z)) \Big) \, ds \, dz
  + \int \vert v \vert^2  \mathring{f} (z)\, dz \\
  &\quad {}- \iiint_0^t
  \frac{d}{ds}
  \Bigg(\frac{f(s, Z(s,0,z))\,f(s, Z(s,0,\tilde z))}
       {|X(s,0,z)-X(s,0,\tilde z)|} \Bigg)\,ds\, dz\, d\tilde{z} \\
  &\quad {}
  -  \iint
  \frac{\mathring{f} (z)\, \mathring{f} (\tilde z)}{|x - \tilde{x}|}
  dz\, d\tilde z \\
  &=  2 E_{\mathrm{kin}}(\mathring{f}) + 2 E_{\mathrm{pot}}(\mathring{f}) \\
  &\quad {} - 2\iint_0^t V(s,0,z) \cdot F(s,X(s,0,z))\,
  f(s, Z(s,0,z))\, ds\, dz \\
  &\quad
  {}+ \iiint_0^t
  \frac{X(s,0,z)-X(s,0,\tilde z)}{|X(s,0,z)-X(s,0,\tilde z)|^3}
  \cdot \big(V(s,0,z)-V(s,0,\tilde z)\big)\,\\
  & \qquad\qquad\qquad\qquad\qquad\qquad\qquad
  f(s, Z(s,0,z))\, f(s,Z(s,0,\tilde z))\, ds\, dz\, d\tilde z;
\end{align*}
in the last step we used that $f$ is a strong Lagrangian solution so that 
\begin{align*}
  \frac{d}{ds} f(s,Z(s,0,z)) =0, \quad s \in [0,T[.
\end{align*}
Using Fubini's theorem and reversing the change of variables via
$z \mapsto Z(0,s,z)$ and $\tilde z \mapsto Z(0,s,\tilde z) $,  we obtain
\begin{align}\label{edot}
  E_{\mathrm{kin}}(f(t))  + E_{\mathrm{pot}}(f(t))
  =&  E_{\mathrm{kin}}(\mathring{f}) + E_{\mathrm{pot}}(\mathring{f})
  - \int_0^t \int v \cdot F(s,x)\, f(s,z)\, dz\, ds \nonumber\\
  &{}+
  \frac{1}{2}\int_0^t \!\iint
  (v-\tilde v)\!\cdot\!\frac{x-\tilde x}{|x-\tilde x|^3}
  f(s,\tilde z)\, f(s,z)\, d\tilde z\, dz\, ds.
\end{align}
Since $f$ is a strong Lagrangian solution,
\begin{align*}
  F(s,x)= \int \frac{x-\tilde x}{|x-\tilde x|^3} f(s,\tilde z)\, d\tilde z
\end{align*}
for all $(s,x) \in [0,T[ \times {\mathbb R}^3$ so that the two integrals
in \eqref{edot} cancel,
and the proof of the conservation laws and of Theorem~\ref{main} is complete.
\subsection{Further solution properties and remarks}
\begin{itemize}
\item[(a)]
  The above proof shows that the conservation laws hold for any strong
  Lagrangian solution; the symmetry assumption on the data did not enter
  in that argument.
\item[(b)]
  In none of the preceding arguments did we use the attractive nature
  of the force field so that our results hold for the plasma physics
  case as well,
  where the sign in the right hand side of \eqref{poisson} is reversed.
  In the relativistic, plasma physics case the estimates in \cite{GS},
  which equally well apply to strong Lagrangian solutions,
  imply that $T=\infty$,
  i.e., the solutions are global (for spherically symmetric data).
\item[(c)]
  The proof of Theorem~\ref{main} implies that for all $t\in [0,T[$
  the functions
  $f(t)$ and $\rho(t)$ are bounded, measurable functions with compact support,
  and the control on the support is locally uniform in $t$.
  In the non-relativistic
  case the bound on the velocity support of $f(t)$ is globally uniform in $t$,
  which follows from the estimates in Horst \cite{Ho},
  cf.\ \cite[Thm.~1.4]{Rein07}.
\item[(d)]
  Our existence proof is completely constructive and does not
  rely on compactness
  arguments, which typically are used for obtaining weak solutions,
  and it covers
  both the non-relativistic and the relativistic case.
  The price to pay for this
  is the unwanted symmetry assumption on the initial data.
\end{itemize}

\end{document}